\documentclass[12pt]{article}
 \usepackage{amsthm,amsmath,amssymb}
 \usepackage{color}
 
   \usepackage[all]{xy}
  \usepackage{graphicx}

\theoremstyle{plain}
\newtheorem{Thm}{Theorem}

\newtheorem{lemma}[Thm]{Lemma}
\newtheorem{Prop}[Thm]{Proposition}
\newtheorem{Def}[Thm]{Definition}

\newtheorem{remark}[Thm]{Remark}

\newcommand{\twotwo}{\mathbf{2+2}}
\newcommand{\threeone}{\mathbf{3+1}}

\date{July 25,  2017 }

\title{Interval Orders with Two Interval Lengths}

\author{Simona Boyadzhiyska\\
\small Berlin Mathematical School \\
 \small Freie Universit\"{a}t Berlin \\
\small Berlin, Germany \\
\small\tt \  s.boyadzhiyska@fu-berlin.de
\and
Garth Isaak\\
\small Department of Mathematics\\
\small Lehigh University\\
\small Wellesley MA 02481\\
\small\tt  gi02@lehigh.edu 
\and
Ann  N. Trenk\thanks{
This work was supported by a grant from the Simons Foundation (\#426725, Ann Trenk). } \\
\small Department of Mathematics\\
\small Wellesley College\\
\small Wellesley MA 02481\\
\small\tt atrenk@wellesley.edu
}

\begin{document}

\maketitle

\begin{abstract} 
A poset $P = (X,\prec)$ has an interval representation if each $x \in X$ can be assigned a real interval $I_x$ so that $x \prec y$ in $P$ if and only if $I_x$ lies completely to the left of $I_y$.    Such orders are called \emph{interval orders}.   In this paper we give a surprisingly simple forbidden poset characterization of those posets that have an interval representation in which each interval length is either 0 or 1.    In addition, for posets $(X,\prec)$ with a weight of 1 or 2 assigned to each point, we characterize those that have an interval representation in which for each $x \in X$ the length of the interval assigned to $x$ equals the weight assigned to $x$.  For both these problems we can determine in polynomial time whether the desired interval representation is possible and in the affirmative case, produce such a representation.
\end{abstract}

\bibliographystyle{plain} 


\section{Introduction}

\subsection{Posets and Interval Orders}
A \emph{poset} $P$ consists of a set $X$ of \emph{points} and a relation $\prec$ that is irreflexive and transitive, and therefore antisymmetric.   We consider only posets in which $X$ is a finite set.   It is sometimes convenient to write $y \succ x$ instead of $x \prec y$.    If $x \prec y$ or $y \prec x$ we say that $x$ and $y$ are \emph{comparable}, and otherwise we say they are \emph{incomparable}, and denote the incomparability by $x \parallel y$.   The set of all points incomparable to $x$ is called the \emph{incomparability set of $x$} and  denoted by $Inc(x)$.   An \emph{interval representation} of a poset  $P=(X,\prec)$ is an assignment of a closed real interval $I_v$ to each $v\in X$ so that 
$x \prec y$ if and only if   $I_x$   is completely to the left of $I_y$.   We sometimes denote an interval representation by   ${\mathcal I} = \{I_v: v \in X\}$.  A poset with such a representation is called an \emph{interval order}.  It is well-known that  the classes studied in this paper are the same if open intervals are used instead of closed intervals, e.g., see Lemma 1.5 in \cite{GoTr04}.

The poset $\twotwo$ shown in  Figure~\ref{chains-fig}  consists of four elements $\{a,b,x,y\}$ and the only comparabilities are $a \prec b$ and $x \prec y$.  Interval orders have a lovely characterization theorem that was anticipated by Wiener in 1914 (see \cite{FiMo92}) and shown by Fishburn \cite{Fi70}:
Poset $P$ is an interval order if and only if it contains no induced $\twotwo$.

Interval orders can be used to model scheduling problems. For example, a set $X$ of events together with a time interval $I_x$ for each $x \in X$ produces an interval  order in which $x \prec y$ precisely when event $x$ ends before event $y$ begins.  In some applications, there may be restrictions on the interval lengths in an interval representation.    Posets  that have an interval representation in which all intervals are the same length are known as  \emph{unit interval orders} or \emph{semiorders}.    The poset $\threeone$, consisting of four elements $\{a,b,c,x\}$  whose only comparabilities are $a \prec b \prec c$,  is not a unit interval order.  Indeed, Scott and Suppes \cite{ScSu58} characterize unit interval orders as those posets with no induced $\twotwo$ and no induced $\threeone$.  Figure~\ref{chains-fig} shows the posets $\twotwo$ and  $\threeone$.  

\begin{figure}
\begin{center}
 \begin{picture}(150,60)(0,15)
\thicklines

\put(20,40){\circle*{5}}
\put(20,70){\circle*{5}}

\put(50,40){\circle*{5}}
\put(50,70){\circle*{5}} 
\put(20,40){\line(0,1){30}}
\put(50,40){\line(0,1){30}}

\put(7,38){$a$}
\put(7,68){$b$}

\put(55,38){$x$}
\put(55,68){$y$}

\put(22,10){$\twotwo$}

\put(105,10){$\threeone$}


\put(120,35){\circle*{5}}
\put(120,55){\circle*{5}}
\put(120,75){\circle*{5}}
\put(140,55){\circle*{5}}
\put(120,35){\line(0,1){40}}

\put(107,33){$a$}
\put(107,53){$b$}
\put(107,73){$c$}
\put(145,53){$x$}
  


  \end{picture}
  \end{center}

\caption{The posets $\twotwo$  and   $\threeone$.}

\label{chains-fig}
 \end{figure}

In this paper, we consider interval orders that arise from representations where there are restrictions on the interval lengths.  The classes we consider are between the two extremes of interval orders (no restrictions on interval lengths) and unit interval orders (all intervals the same length).  

\subsection{Digraphs and Potentials}

A {\em directed graph}, or {\em digraph}, is a pair $G=(V,E)$, where $V$ is a finite set of {\em vertices}, and $E$ is a set of ordered pairs $(x,y)$ with $x,y\in V$, called {\em arcs}.   A \emph{weighted digraph} is a digraph in which each arc $(x,y)$ is assigned a real number weight denoted by  $wgt(x,y)$ or  $w_{xy}$.  We sometimes denote the arc $(x,y)$ by $x\rightarrow y$, and in a weighted digraph, by $x \xrightarrow{w_{xy}}y$.  A \emph{potential function}  $p:V\rightarrow \mathbb{R}$, defined on the vertices of a weighted digraph, is a function satisfying $p(y) - p(x) \leq w_{xy}$ for each arc  $(x,y)$.  Theorem~\ref{potential-no-neg} is a well-known result that   specifies precisely which digraphs have potential functions.

An \emph{$xy$-walk} in a digraph $G$ is a sequence of vertices  $W:   x_1, x_2,\dots, x_{t-1}, x_t  $ so that  $x = x_1$, $y = x_t$, and $(x_{i}, x_{i+1})$  is an arc for $i = 1,2,3, \ldots, t-1$.  A \emph{cycle} in digraph $G$ is a sequence of distinct vertices $C:  x_1, \dots, x_{t-1}, x_t $  so that $(x_{i}, x_{i+1})$  is an arc for $i = 1,2,3, \ldots, t-1$ and $(x_t,x_1)$ is also an arc of $G$.  

The \emph{weight} of a walk or cycle in a weighted digraph is the sum of the weights of the arcs it includes.   We write $wgt(C)$ to denote the weight of cycle $C$.     A cycle with negative weight is called a \emph{negative cycle}.  The following theorem is well-known, see Chapter 8 of \cite{Sc03}, for example.  We provide a proof  for completeness and because the proof provides  part of an algorithm for producing a potential function in a weighted digraph with no negative cycle.

\begin{Thm}
A weighted digraph has a potential function if and only if it contains no negative cycle.
\label{potential-no-neg}
\end{Thm}

\begin{proof}
Suppose weighted digraph $G$ has a potential function $p$.  If $G$ contains a negative weight cycle $C:   x_1, \dots, x_{t-1}, x_t $, then summing the inequalities $p(x_{i+1}) - p(x_i) \le w_{x_{i}x_{i+1}}$ for $i = 1, 2, 3, \ldots, t-1$ with the inequality $p(x_1) -p(x_t) \le w_{x_t x_1}$ yields   $0 \le wgt(C)$, a contradiction.

Conversely,  suppose a weighted digraph $G$ contains no negative cycle.  For each vertex $y \in V(G)$, let $\overline{p}(y)$ be the minimum weight of a walk ending at $y$.  Since $G$ is finite and has no negative cycles, the values $\overline{p}(y)$ are well-defined and we need only consider walks with distinct vertices e.g., paths.    It remains to show that the function $\overline{p}$ is a potential function on $G$.  Consider any arc $(x,y)$ in $G$.  Any minimum weight path ending at $x$ followed by the arc $(x,y)$ creates a path ending at $y$ with weight $\overline{p}(x) + w_{xy}.$  Thus by  the definition of $\overline{p}$ we have $\overline{p}(y) \le \overline{p}(x) + w_{xy}.$
\end{proof}

\subsection{Weighted Posets and  Related Digraphs  }

A \emph{weighted poset}   $(P,f)$ consists of a poset $P= (X,\prec)$ together with a weight function $f$ from $X$ to the non-negative reals.    We are interested in the case in which $P$ is an interval order and we seek an  interval representation of $P$ in which $f(x)$ is the length of the interval assigned to $x$ for each $x \in X$.    We restrict the range of $f$ to be the set $\{0,1\}$ in Section~\ref{0-1-sec} and $\{1,2\}$ in Section~\ref{1-2-sec}.
 Given a weighted  poset $(P,f)$, we construct a weighted digraph $G(P,f)$    in Definition~\ref{GPf-def}  and show in Proposition~\ref{potential-prop} that it has  the following property:    $P$ has an interval representation ${\mathcal I} = \{I_x: x \in X\}$  in which $|I_x| = f(x)$ for all $x \in X$ if and only if $G(P,f)$ has no negative cycles.    
We choose the value of  $\epsilon$  appearing as a weight in $G(P,f)$  so that $0 < \epsilon < \frac{1}{|X|^2}$.   We also define the closely related digraph $G'(P,f)$, which eliminates the constant $\epsilon$ and simplifies our arguments.
  
  \begin{Def} {\rm
 Let $(P,f)$ be a weighted poset and $P=(X,\prec)$.   Define $G(P,f)$ to be the weighted digraph with vertex set $X$ and the following arcs.
 
\begin{itemize}
 \item $(a,b )$ with weight $w_{ab} = -f(b) - \epsilon$ for all $a,b \in X$ with $a \succ b$.

  \item $(a,b)$ with weight $w_{ab} = f(a)$  and $(b,a)$ with weight $w_{ba} = f(b)$ for all distinct $a,b \in X$ with $a \parallel b$.
\end{itemize}
The digraph $G'(P,f)$ is identical to $G(P,f)$ except that for $a \succ b$, the arc $(a,b)$ has weight $w'_{ab} = -f(b)$.    }
 \label{GPf-def}
 \end{Def}

We classify   arc  $(a,b)$ of $G(P,f)$  or  $G'(P,f)$ as $(-)$ if  $a \succ b$ (even when $f(b) = 0$) and  as $(+)$ if $a \parallel b$ (even when $f(a)=0)$).  Likewise, we classify paths in $G(P,f)$ and $G'(P,f)$  by their arc types, for example the path 
$a \xrightarrow{-f(b)- \epsilon}b\xrightarrow{f(b)}c \xrightarrow{f(c)}d$ in $G(P,f)$, which corresponds to $a \succ b \parallel c \parallel d$ in $P$, would be classified as $(-,+,+)$.     We sometimes find it convenient to specify the start and end of a path, and if a path $S$ starts at  point $a$ and ends at point $b$ we may write it as ${_aS_b}$ or simply as $S$.
Likewise, we denote the segment of a cycle $C$ that starts at point $a$ and ends at point $b$ by ${_aC_b}$.

 For easy reference, we list the arcs of $G(P,f)$ and $G'(P,f)$  by category.

\begin{center}
\begin{tabular}{|l|c|c|c|c|}
\hline
Type & Arc & Weight in $G(P,f)$ & Weight in $G'(P,f)$ & $x,y$ Relation \\ 
\hline
$(-)$ & $(a,b)$  & $-f(b) -\epsilon $ &  $-f(b)  $ & $a \succ b$ \\
\hline
$(+)$ & $(a,b)$  & $f(a)$  & $f(a)$ & $a \parallel b$ \\
\hline
 $(+)$ & $(b,a)$  & $f(b)$ &  $f(b)$ & $a \parallel b$ \\
\hline
\end{tabular}
\end{center}

 The following proposition uses Definition~\ref{GPf-def} and the definition of $\epsilon$ to show that a cycle has negative weight in $G(P,f)$ if and only if it contains at least one  $(-)$ arc and  has weight at most 0 in $G'(P,f)$.
 
 \begin{Prop}  Let   $P= (X,\prec)$  be a poset and  $f$  be a weight function from $X$ to the non-negative reals.   
 Digraph $G(P,f)$ has a negative weight cycle if and only if $G'(P,f)$ has a cycle with at least one $(-)$ arc and with weight at most 0.	 
 \label{G'-rem}
 \end{Prop}
 
 \begin{proof}
 If $C$ is a cycle in $G'(P,f)$ that has at least one $(-)$ arc and weight at most 0, then by Definition~\ref{GPf-def}, cycle $C$ has negative weight in $G(P,f)$. 
 
 Conversely, let  $C$ be a negative weight cycle in $G(P,f)$.  Since $(+)$ arcs have weight at least 0, we know that $C$ contains a $(-)$ arc.   By construction, the digraph $G(P,f)$ has fewer than $|X|^2$ arcs, thus $C$ has fewer than $|X|^2$ arcs.    By our choice of $\epsilon$, we can write $wgt(C)$ in $G(P,f)$ as $-M-k\epsilon$ (where $M$ is a non-negative integer and $k\epsilon < 1$).  Then $wgt(C)$ in $G'(P,f)$ is  $-M$, which is at most 0 because $-M-k\epsilon< 0$.
 \end{proof}

 



  The next proposition shows the utility of the digraph $G(P,f)$ in determining whether poset $P$ has an interval representation in which for all points $x$ the length of the interval assigned to $x$ is $f(x)$.

\begin{Prop}  Let $P= (X,\prec)$ be a poset and $f$ a function from $X$ to the non-negative real numbers.  
Poset  
$P$ has an interval representation ${\mathcal I} = \{I_x: x \in X\}$  in which $|I_x| = f(x)$ for all $x \in X$ if and only if $G(P,f)$ has no negative cycles.    
\label{potential-prop}
\end{Prop}

\begin{proof}
\noindent ($\Longrightarrow$)
Fix an interval representation of $P$ in which for each $x \in X$, the interval assigned to $x$ has left endpoint $L(x)$ and length $f(x)$.  Thus the interval assigned to $x$ is  $[L(x),R(x)]$ where $R(x) = L(x) + f(x)$.  Choose $\epsilon$ to be a positive real number less than the smallest distance between distinct endpoints in this representation and  also less than $|X|^2$.  Let $G(P,f)$ be  the resulting weighted digraph.           We show $L(b) - L(a) \le w_{ab}$ for each arc   $(a,b)$  of $G(P,f)$.

If $a \succ b$ then $R(b) < L(a)$ so by our choice of $\epsilon$ we have, $R(b) \le L(a) - \epsilon$.  Then $L(b) + f(b) = R(b)  \le L(a) - \epsilon$ and thus $L(b) - L(a) \le -f(b) - \epsilon = w_{ab}$.    If $a \parallel b$ then $R(a) \ge L(b)$ so $L(a) + f(a)  \ge L(b)$ or equivalently $L(b) - L(a) \le f(a) = w_{ab}$.  Thus $L$ is a potential function for $G(P,f)$.  By Theorem~\ref{potential-no-neg}, digraph
$G(P,f)$ has no negative cycles.    

\medskip
\noindent ($\Longleftarrow$)
Conversely, suppose  $G(P,f)$ has no negative cycles.  By Theorem~\ref{potential-no-neg}, the digraph $G(P,f)$ has a potential function, call it $L$.    For each $x \in X$, let $I_x = [L(x), L(x) + f(x)]$ and note that $I_x$ is indeed an interval since $f(x) \ge 0$.    Using the definitions of $G(P,f)$ and of potential functions, one can check that the set of intervals $\{I_x: x \in X\}$ gives an interval representation of $P$ in which $|I_x| = f(x)$ for each $x$.
\end{proof}

\subsection{The Minimality Hypothesis}

In the next sections our proofs will  involve a cycle $C$ in    digraph $G'(P,f)$  that satisfies a minimality condition.  The next definition makes this precise.

\begin{Def} {\rm 
 Let $P$ be an interval order with $P = (X,\prec)$ and let $f:X \to \{0, 1,2, \ldots\}$ be a weight function.    We say that   cycle $C$ in  $G'(P,f)$ satisfies the \emph{minimality hypothesis} for $(P,f)$ if  $wgt(C) \le 0$, $C$ contains at least one $(-)$ arc, and $C$   has the  minimum number of arcs among such cycles.  }
 \label{min-hyp-def}
\end{Def}

We end this section with a lemma  that establishes properties of cycles that satisfy the minimality hypothesis for $(P,f)$.

\begin{lemma}
 Let $r$ be a positive integer.
If $C$ satisfies the minimality hypothesis for $(P,f)$ and $f(x) \le r$ for each  point $x$ of $P$
then $wgt(C) \ge  1-r$. 
\label{wgt-C-lem}
\end{lemma}
 
 \begin{proof}
 The arc weights of $G'(P,f)$ are integers,  thus we may suppose for a contradiction that $wgt(C) \le -r$.       Since  $C$ contains a $(-)$ arc and any path in $G'(P,f)$ of the form $(-,+)$ has weight $0$, $C$ must contain a segment of the form $(-,-)$.  Thus there exist vertices $a,b,c$ so that $S: a \xrightarrow{ -f(b)}b\xrightarrow{-f(c) }c$ is  a segment of $C$.  By the definition of $G'(P,f)$ we have $a \succ b \succ c$ in $P$ and thus we can replace $S$ by  the $(-)$ arc $a \xrightarrow{ -f(c)}c$ to obtain a shorter cycle $C'$ in $G'(P,f)$ with $wgt(C') = wgt(C) + f(b) \le -r +r \le 0$.  This contradicts the minimality of $C$.
\end{proof}

\section{Interval Orders Representable with Lengths 0 or 1}
\label{0-1-sec}
 
 We say that a poset has a \emph{$\{0,1\}$-interval representation} if it has an interval representation in which each interval  has length either 0 or 1. We use weighted digraphs to characterize this class and our forbidden poset characterization contains just four posets.     Rautenbach and Szwarcfiter \cite{RaSz12} have characterized the analogous class of interval graphs, however the  characterization  in the graph setting is more complicated. We can derive our characterization from the graph version, however, that derivation is more involved than a direct order-based proof.
    
A \emph{simplicial} vertex in a graph is one whose neighbor set forms a clique.   An \emph{antichain} in a poset is a set of points for which every pair is incomparable.  For example, the set $\{d,b,x\}$ is an antichain in all four posets shown in Figure~\ref{forb-fig}.
 We introduce the term \emph{co-simplicial} in Definition~\ref{co-simp-def} so that 
a point $v$ is co-simplicial in poset $P$ if and only if $v$ is simplicial in  the incomparability graph  of $P$.  
\begin{Def}{\rm
A point in a poset is \emph{co-simplicial} if its incomparability set is an antichain.
}
\label{co-simp-def}
\end{Def}
In Figure~\ref{forb-fig}, the point $d$ is co-simplicial in the first and last poset shown and not co-simplicial in the middle two posets.  
  In the next lemma, we will show that  in any $\{0,1\}$-interval representation of a poset, points that are not co-simplicial \emph{must} be assigned intervals of length 1.  Furthermore, if poset $P$ has a $\{0,1\}$-interval representation,   co-simplicial points \emph{can} be assigned intervals of length 0.

\begin{lemma}
If a poset $P$ has a $\{0,1\}$-interval representation then $P$ has a $\{0,1\}$-representation in which $|I_x| = 0$ for each co-simplicial point $x$ and $|I_x|= 1$ for each point $x$ that is not co-simplicial.
\label{length-lem}
\end{lemma}

\begin{proof}
Suppose $y$ is  point of poset $P$ that is not co-simplicial.  Thus there exist points $a,b \in Inc(y)$ for which $a \prec b$.  In any interval representation of $P$, the interval $I_a$ assigned to $a$ lies completely to the left of the interval $I_b$ assigned to $b$, and the interval $I_y$ assigned to $y$ must intersect both $I_a$ and $I_b$.   Thus $|I_y| \neq 0$.  This proves that in any $\{0,1\}$-interval representation of a poset, points that are not co-simplicial \emph{must} be assigned intervals of length 1.

  Now fix a $\{0,1\}$-interval representation of $P$ and let $x$ be a co-simplicial point of $P$.  By Definition~\ref{co-simp-def}, $Inc(x)$ is an antichain.  Thus for all $u,v \in Inc(x)$, we have
 $I_u \cap I_v \neq \emptyset$ and by the Helly property of intervals, 
 $\bigcap\limits_{v\in Inc(x)} I_v \neq \emptyset$. Hence $I_x$ can be contracted to a single point in this intersection.   Repeat this argument for each co-simplicial point of $P$ until each is assigned an interval of length 0.  
\end{proof}

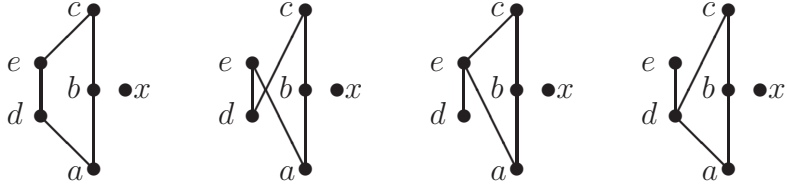
\begin{figure}
\begin{center}
 \begin{picture}(300,65)(0,15)
\thicklines

\put(20,40){\circle*{5}}
\put(20,60){\circle*{5}}

\put(40,20){\circle*{5}}
\put(40,50){\circle*{5}}
\put(40,80){\circle*{5}}
\put(52,50){\circle*{5}}
\put(20,40){\line(0,1){20}}
\put(40,20){\line(0,1){60}}
\put(40,20){\line(-1,1){20}}
\put(20,60){\line(1,1){20}}

\put(7,37){$d$}
\put(7,57){$e$}
\put(30,16){$a$}
\put(30,47){$b$}
\put(30,78){$c$}
\put(55,47){$x$}

\put(100,40){\circle*{5}}
\put(100,60){\circle*{5}}

\put(120,20){\circle*{5}}
\put(120,50){\circle*{5}}
\put(120,80){\circle*{5}}
\put(132,50){\circle*{5}}
\put(100,40){\line(1,2){20}}
\put(120,20){\line(-1,2){20}}
\put(120,20){\line(0,1){60}}
\put(100,40){\line(0,1){20}}

\put(87,37){$d$}
\put(87,57){$e$}
\put(110,16){$a$}
\put(110,47){$b$}
\put(110,78){$c$}
\put(135,47){$x$}

\put(180,40){\circle*{5}}
\put(180,60){\circle*{5}}

\put(200,20){\circle*{5}}
\put(200,50){\circle*{5}}
\put(200,80){\circle*{5}}
\put(212,50){\circle*{5}}
\put(180,60){\line(1,1){20}}
\put(200,20){\line(-1,2){20}}
\put(200,20){\line(0,1){60}}
\put(180,40){\line(0,1){20}}

\put(167,37){$d$}
\put(167,57){$e$}
\put(190,16){$a$}
\put(190,47){$b$}
\put(190,78){$c$}
\put(215,47){$x$}

\put(260,40){\circle*{5}}
\put(260,60){\circle*{5}}

\put(280,20){\circle*{5}}
\put(280,50){\circle*{5}}
\put(280,80){\circle*{5}}
\put(292,50){\circle*{5}}
\put(260,40){\line(1,2){20}}
\put(280,20){\line(-1,1){20}}
\put(280,20){\line(0,1){60}}
\put(260,40){\line(0,1){20}}

\put(247,37){$d$}
\put(247,57){$e$}
\put(270,16){$a$}
\put(270,47){$b$}
\put(270,78){$c$}
\put(295,47){$x$}

  \end{picture}
  \end{center}

\caption{The set ${\cal H}$ of minimal forbidden posets for interval orders with a $\{0,1\}$-representation.}

\label{forb-fig}
 \end{figure}
 
\begin{lemma} Let $P$ be an interval order $(X,\prec)$ and  
 $f: X \to \{0,1\}$ be the function defined by $f(x) = 0$ when $x$ is co-simplicial in $P$ and $f(x) =1$ otherwise.  Let $C$ be a cycle in $G'(P,f)$ satisfying the minimality hypothesis.
If $C$ contains a segment  $S_1: a  \xrightarrow{-f(b)  }b  \xrightarrow{-f(c)   }c$ then $f(b) = 1$, and if $C$ contains a segment  $S_2: b  \xrightarrow{-f(c) }c  \xrightarrow{ +f(c)  }d$ then $f(d) = 1$.
\label{bd-weights-lem}
\end{lemma}

\begin{proof}
First suppose that  $C$ contains the segment $S_1$,  and thus $a \succ b \succ c$ in $P$.  If $f(b) \neq 1$ we would have $f(b) = 0$ and could replace $S_1$ by segment $a  \xrightarrow{-f(c)  }c$ to obtain a shorter cycle that has a $(-)$ arc and  with the same weight as $C$, a contradiction.   

Next suppose that $C$ contains the segment $S_2$.  Thus in $P$ we have $b \succ c $ and $c \parallel d$.
If $b \prec d$ we get $c \prec b \prec d  $, contradicting $c \parallel d$.   If $b \succ d$, we can replace the segment $S_2$ by segment $b  \xrightarrow{-f(d)  }d$ to obtain a shorter cycle with   weight at most 0.  Hence,    
 $b \parallel d$ and thus $b, c \in Inc(d)$.  In this case, $d$ is not co-simplicial, so $f(d) = 1$ as desired.  
 \end{proof}

\begin{Thm}
Let $P$ be an interval order and define function $f$  by $f(x) = 0$ when $x$ is co-simplicial in $P$ and $f(x) =1$ otherwise.
The following are equivalent.
\begin{enumerate}
\item  $P$ has a $\{0,1\}$-interval representation.
\item  For every $\threeone$ induced in $P$, the middle element of the chain is co-simplicial.
\item  Digraph $G(P,f)$ has no negative cycles. 

\item  Every cycle in digraph $G'(P,f)$  with at least one $(-)$ arc has positive weight.

\item  $P$ does not contain any induced poset from the set $\cal H$ (shown in Figure~\ref{forb-fig}).
\end{enumerate}
\label{zero-one-thm}
\end{Thm}

\begin{proof}
\noindent $ (1) \Longrightarrow (2)$.    We are given that $P$ has a $\{0,1\}$-interval representation, and by Lemma~\ref{length-lem},  we may fix a $\{0,1\}$-interval representation $\mathcal I$ in which points that are not co-simplicial get length 1.  Let $I_v$ be the interval assigned to point $v$ in $\mathcal I$.  For a contradiction, assume that the $\threeone$ $(a \prec b \prec c) \parallel x$ is induced in $P$ and that $b$ is not co-simplicial.  Hence $|I_b| = 1$.  Since $a \prec b \prec c$, the intervals $I_a$, $I_b$, and $I_c$ are disjoint with $I_b$ between $I_a$ and $I_c$.    However,    $x \parallel a$ and $x \parallel c$  so $|I_x| > |I_b| = 1$,  contradiction.
\medskip

\noindent $ (2) \Longrightarrow (5)$.    We prove the contrapositive.  In each of the posets in $\cal H$, the elements $a,b,c,x$ induce a $\threeone$ and the middle element $b$ of the chain $a \prec b \prec c$ is  not co-simplicial.
\medskip

\noindent $ (5) \Longrightarrow (2)$.    We  again prove the contrapositive.  Suppose there exists $\threeone$ $(a \prec b \prec c) \parallel x$ induced in $P$ for which $b$ is not  co-simplicial.  By Definition~\ref{co-simp-def}, there exist points  $d,e \in Inc(b)$ for which $d \prec e$.  If $x=d$ then the elements $(b \prec c) \parallel (x \prec e)$ induce a $\twotwo$ in $P$, a contradiction since $P$ is an interval order.  Similarly, $x=e$ leads to a contradiction.  Thus $a,b,c,d,e,x$ are six distinct elements of $P$.  We must have $d \prec c$,  for otherwise $(d \prec e) \parallel (b \prec c)$ form an induced $\twotwo$.  Similarly, we must have $a \prec e$.  If $x \prec e$ the elements $(x \prec e) \parallel (b \prec c)$ form a $\twotwo$, and similarly if $d \prec x$, the elements $(d \prec x) \parallel (a \prec b)$ form a $\twotwo$,  both leading to contradictions.  If $e \prec x$ then $a \prec e \prec x$, contradicting $a \parallel x$, and similarly $x \prec d$ leads to a contradiction.   Thus $d \parallel x$ and $e \parallel x$.    There are only two relations that are not determined:  that between $e$ and $c$ (namely $e \prec c$ or $e \parallel c$) and that between $d$ and $a$ (namely $a \prec d$ or $a \parallel d$).  The four possible combinations of these relations lead to the four posets in $\cal H$, so one of the posets in $\cal H$ is induced in $P$, a contradiction.

\medskip

\noindent $ (3) \Longrightarrow (1)$.
This follows immediately from Proposition~\ref{potential-prop}.


\medskip

\noindent $ (4) \Longrightarrow (3)$.
The contrapositive follows immediately from Proposition~\ref{G'-rem}.

\medskip
\noindent $ (2) \Longrightarrow (4)$.
We are given that for every $\threeone$ induced in $P$, the middle element of the chain is co-simplicial and we wish to show that  every cycle of $G'(P,f)$  with at least one $(-)$ arc has positive weight.  For a contradiction, assume that $G'(P,f)$ has a    cycle with at least one $(-)$ arc and  weight at most 0, and let $C$ be such a cycle with a minimum number of arcs.  Thus $C$ satisfies the minimality hypothesis (Definition~\ref{min-hyp-def}) for $G'(P,f)$.

First we show that  $C$ has at least two arcs of type $(-)$. 
Suppose $C$ has just one arc  $(a,b)$  of type $(-)$ and consider the segment $S: a \xrightarrow{-f(b) }b\xrightarrow{+f(b)}c $ of $C$ with $wgt(S) = 0$.  Since $wgt(C) \le 0$ and  $f(x) \in  \{0,1\}$ for each $x \in X$, the remaining arcs each have weight 0 and $wgt(C) = 0$.   By Lemma~\ref{bd-weights-lem},
$f(c) = 1$.  Then the arc on $C$ leaving $c$ has weight $f(c)$ with $f(c) = 1 > 0$, a contradiction.   Thus $C$ has at least two arcs of type $(-)$.
   
   We next show that $C$ does not contain a segment of type $(+,-,+)$.  Suppose $C$ contains 
    a segment $S_1: a \xrightarrow{ +f(a)}b\xrightarrow{-f(c) }c  \xrightarrow{+f(c)}d$.  By definition of $G'(P,f)$ we have $a \parallel b$, $b \succ c$, and $c \parallel d$.    If $a \prec d$, the elements $a,b,c,d$ induce a $\twotwo$ in $P$ a contradiction since $P$ is an interval order.    Thus either $a \succ d$ or $a \parallel d$.  In these cases, we can  replace the segment $S_1$  by $a \xrightarrow{  }d   $ to obtain a shorter cycle $C'$   with  $wgt(C') \le wgt(C) \le 0$.  Since $C$ contains at least two $(-)$ arcs, $C'$ still contains a $(-)$ arc,  contradicting the minimality of $C$.    
   
   Since $C$ contains at least two $(-)$ arcs, one $(+)$ arc, and no segment of the form $(+,-,+)$, cycle $C$ must contain a segment of type   $(-,-,+)$.  Let   $S_2: a \xrightarrow{ -f(b)   }b\xrightarrow{-f(c) }c  \xrightarrow{f(c)}d$ be such a segment.  By Lemma~\ref{bd-weights-lem}, we have $f(b) = 1$ and $f(d) = 1$.    
 If $ d \succ a $ we get $ d \succ a \succ b \succ c $, contradicting $c \parallel d$.  If $a \succ d$ we can replace $S_2$ by the $(-)$ arc $a \xrightarrow{ -f(d)  }d$ to obtain a shorter cycle with negative weight, a contradiction.  Hence, $a \parallel d$ and the points $a,b,c,d$ induce a $\threeone$ in $P$.  By the hypothesis,    the middle element of the chain, $b$, is co-simplicial and  by the definition of $f$ we have $f(b) =0$.  This contradicts our earlier conclusion that $f(b) = 1$.
\end{proof}


We close this section by briefly describing how to construct a $\{0,1\}$-interval representation of a poset $P$ or
produce one of the forbidden induced posets algorithmically. We use a standard shortest-paths algorithm such as the Bellman-Ford or the matrix multiplication method on $G(P,f)$ to compute the weight of a minimum-weight path between each pair of vertices or detect a negative cycle.  If there is a negative cycle, these algorithms detect one with a minimum number of arcs, corresponding to a forbidden induced poset from Theorem~\ref{zero-one-thm}.
If there is no negative cycle, we construct a $\{0,1\}$-interval representation of $P$ as described in the proofs of  Proposition~\ref{potential-prop} and Theorem~\ref{potential-no-neg}.
Thus there is a polynomial time certifying algorithm. 


\section{Interval Orders Representable with Lengths 1 or 2}
\label{1-2-sec}
   In this section,  we consider posets that have an interval representation in which the interval lengths are 1 or 2.   Unlike in Section~\ref{0-1-sec}, 
   we are not aware of any result characterizing the analogous class of interval graphs.  We also  do not have an analogue of Lemma~\ref{length-lem},
   which allowed us to determine the length of the interval assigned to each point in the case of $\{0,1\}$-representations.    Instead, we consider weighted posets $(P,f)$ where $P = (X,\prec)$ and $f: X \to \{1,2\}$ is a weight function.   We determine which have interval
      representations ${\mathcal I} = \{I_x: x \in X\}$  in which $|I_x| = f(x)$ for all $x \in X$. 
   

 In Theorem~\ref{main-thm} we characterize this set of weighted posets
  as those with no induced weighted poset in a set  $\cal F$.  The set $\cal F$,  defined formally in Definition~\ref{forb-def},
  consists of the  poset $\threeone$  with  weightings shown in Figure~\ref{three-one-forb},   together with the  four infinite families illustrated  in the case of $t=6$ in Figure~\ref{forb-2-fig}. 

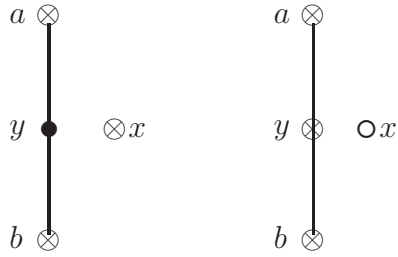
\begin{figure}
\begin{center}

 \begin{picture}(200,90)(0,0)
\thicklines
\put(35,5){$\otimes$}
\put(40,50){\circle*{6}}
\put(60,47){$\otimes$}
\put(35,90){$\otimes$}

\put(25,5){$b$}
\put(25,48){$y$}
\put(70,47){$x$}
\put(25,90){$a$}

\put(40,10){\line(0,1){80}}

\put(135,5){$\otimes$}
\put(135,47){$\otimes$}
\put(160,50){\circle{6}}
\put(135,90){$\otimes$}

\put(125,5){$b$}
\put(125,48){$y$}
\put(165,47){$x$}
\put(125,90){$a$}

\put(140,10){\line(0,1){80}}

  \end{picture}
  \end{center}

\caption{Forbidden weighted posets. Solid circles represent points with weight 2, hollow circles represent points with weight 1, and circles with a cross represent points whose weight may be either 1 or 2.}  

\label{three-one-forb}
 \end{figure}

 \begin{Def} { \rm
 The set $\cal F$ consists of the poset $\threeone$ with weights shown in Figure~\ref{three-one-forb} and 
 four infinite families, ${\cal F}_1$, ${\cal F}_2$,  ${\cal F}_3$,  ${\cal F}_4$.     The posets in each   family ${\cal F}_j$ contain  the points $x_1, x_2, \ldots, x_{t+1}, y_0, y_1, \ldots, y_{t+1}, a, b$ and posets in families ${\cal F}_3$ and ${\cal F}_4$ contain the extra point $x_0$.  The following comparabilities, as well as those implied by transitivity, are present in each family for $i \ge 0$:  \ $b \prec y_0, \  y_i \prec y_{i+1}, \  x_i \prec x_{i+2}, \  x_i \prec y_{i+1}, \  y_i \prec x_{i+2}$.    In each ${\cal F}_j$, $f(a)$ and $f(b)$ may be either 1 or 2.  The remaining weights are 2 with the  exceptions shown in the table below, which also shows additional features specific to each family. 
 
 \bigskip
 
 \noindent
 \begin{tabular}{|c|c|c|c|c|}\hline
 Family & $a=y_{t+1}$?  &  Add'l comparabilities &  $v$ with $f(v) = 1$ &  $x_0$ exists?  \\ \hline
  ${\cal F}_1$ & No & $y_{t+1} \prec a$ &  $y_0$, \ $ y_{t+1}$  &  No  \\ \hline
  ${\cal F}_2$ & Yes &   &  $y_0$, \ $ x_{t+1}$  &  No  \\ \hline
  ${\cal F}_3$ & No & $y_{t+1} \prec a$, \ \  $b \prec x_1$ &  $x_0$, \ $ y_{t+1}$  &  Yes  \\ \hline
  ${\cal F}_4$ & Yes & $b \prec x_1$ &  $x_0$, \ $ x_{t+1}$  &  Yes  \\ \hline
 \end{tabular}
 \medskip
 }
\label{forb-def}
 \end{Def}

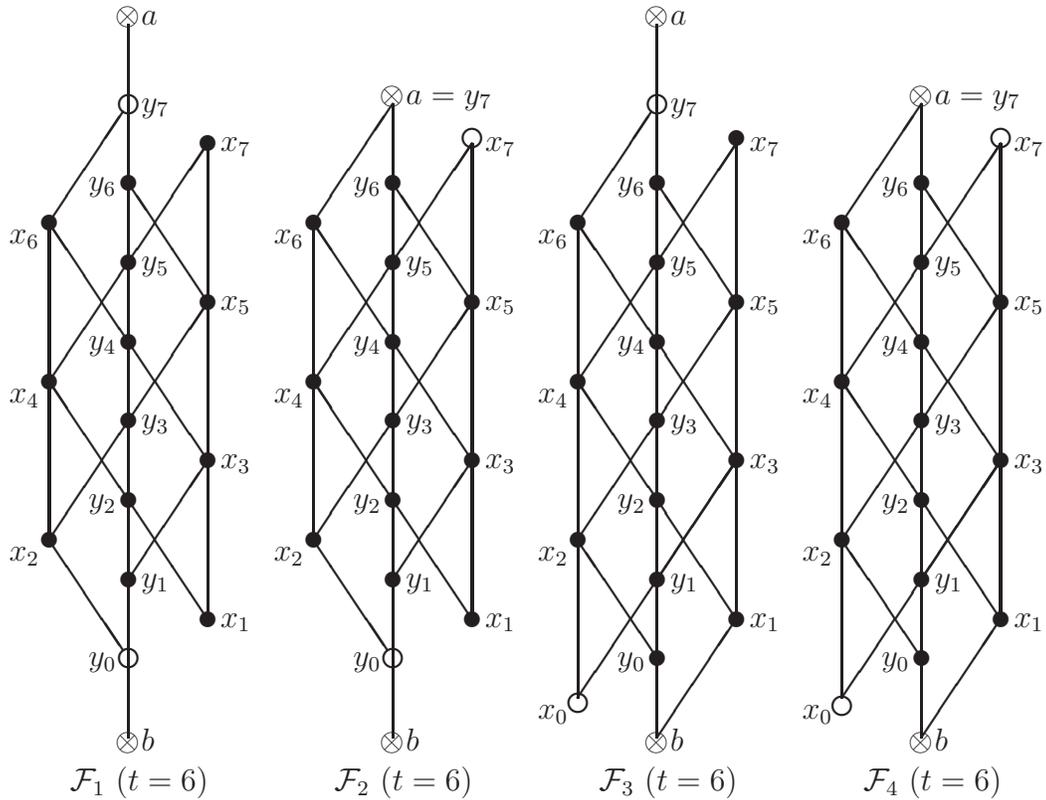
\begin{figure}
\begin{center}
 \begin{picture}(400,300)(0,-10)
\thicklines


\put(20,85){\circle*{6}}
\put(20,145){\circle*{6}}
\put(20,205){\circle*{6}}

\put(45,5){$\otimes$}
\put(50,40){\circle{7}}
\put(50,70){\circle*{6}}
\put(50,100){\circle*{6}}
\put(50,130){\circle*{6}}
\put(50,160){\circle*{6}}
\put(50,190){\circle*{6}}
\put(50,220){\circle*{6}}
\put(50,250){\circle{7}}
\put(45,280){$\otimes$}

\put(80,55){\circle*{6}}
\put(80,115){\circle*{6}}
\put(80,175){\circle*{6}}
\put(80,235){\circle*{6}}

\put(20,85){\line(0,1){120}}
\put(50,10){\line(0,1){270}}
\put(80,55){\line(0,1){180}}
\put(80,55){\line(-2,3){60}}
\put(80,115){\line(-2,3){60}}
\put(80,175){\line(-2,3){30}}
\put(20,85){\line(2,-3){28}}
\put(20,85){\line(2,3){60}}
\put(20,145){\line(2,3){60}}
\put(20,205){\line(2,3){28}}
\put(50,70){\line(2,3){30}}

\put(5,77){$x_2$}
\put(5,137){$x_4$}
\put(5,197){$x_6$}

\put(85,52){$x_1$}
\put(85,112){$x_3$}
\put(85,172){$x_5$}
\put(85,232){$x_7$}

\put(55,5){$b$}
\put(35,37){$y_0$}
\put(55,67){$y_1$}

\put(35,97){$y_2$}
\put(55,127){$y_3$}
\put(35,157){$y_4$}
\put(55,187){$y_5$}
\put(35,217){$y_6$}
\put(55,247){$y_7$}
\put(55,280){$a$}

\put(28,-10){${\cal F}_1$  ($t=6$)}



\put(120,85){\circle*{6}}
\put(120,145){\circle*{6}}
\put(120,205){\circle*{6}}

\put(145,5){$\otimes$}
\put(150,40){\circle{7}}
\put(150,70){\circle*{6}}
\put(150,100){\circle*{6}}
\put(150,130){\circle*{6}}
\put(150,160){\circle*{6}}
\put(150,190){\circle*{6}}
\put(150,220){\circle*{6}}
\put(145,250){$\otimes$}

\put(180,55){\circle*{6}}
\put(180,115){\circle*{6}}
\put(180,175){\circle*{6}}
\put(180,237){\circle{7}}

\put(120,85){\line(0,1){120}}
\put(150,10){\line(0,1){240}}
\put(180,55){\line(0,1){180}}
\put(180,55){\line(-2,3){60}}
\put(180,115){\line(-2,3){60}}
\put(180,175){\line(-2,3){30}}
\put(120,85){\line(2,-3){28}}
\put(120,85){\line(2,3){60}}
\put(120,145){\line(2,3){60}}
\put(120,205){\line(2,3){30}}
\put(150,70){\line(2,3){30}}

\put(105,77){$x_2$}
\put(105,137){$x_4$}
\put(105,197){$x_6$}

\put(185,52){$x_1$}
\put(185,112){$x_3$}
\put(185,172){$x_5$}
\put(185,232){$x_7$}

\put(155,5){$b$}
\put(135,37){$y_0$}
\put(155,67){$y_1$}

\put(135,97){$y_2$}
\put(155,127){$y_3$}
\put(135,157){$y_4$}
\put(155,187){$y_5$}
\put(135,217){$y_6$}
\put(155,250){$a=y_7$}

\put(128,-10){${\cal F}_2$  ($t=6$)}


\put(220,23){\circle{7}}
\put(220,85){\circle*{6}}
\put(220,145){\circle*{6}}
\put(220,205){\circle*{6}}

\put(245,5){$\otimes$}
\put(250,40){\circle*{6}}
\put(250,70){\circle*{6}}
\put(250,100){\circle*{6}}
\put(250,130){\circle*{6}}
\put(250,160){\circle*{6}}
\put(250,190){\circle*{6}}
\put(250,220){\circle*{6}}
\put(250,250){\circle{7}}
\put(245,280){$\otimes$}

\put(280,55){\circle*{6}}
\put(280,115){\circle*{6}}
\put(280,175){\circle*{6}}
\put(280,237){\circle*{6}}

\put(220,25){\line(0,1){180}}
\put(250,10){\line(0,1){270}}
\put(280,55){\line(0,1){180}}
\put(280,55){\line(-2,3){60}}
\put(280,115){\line(-2,3){60}}
\put(280,175){\line(-2,3){30}}
\put(250,40){\line(-2,3){30}}
\put(220,25){\line(2,3){60}}
\put(220,85){\line(2,3){60}}
\put(220,145){\line(2,3){60}}
\put(220,205){\line(2,3){28}}
\put(250,70){\line(2,3){30}}
\put(250,10){\line(2,3){30}}

\put(205,17){$x_0$}
\put(205,77){$x_2$}
\put(205,137){$x_4$}
\put(205,197){$x_6$}

\put(285,52){$x_1$}
\put(285,112){$x_3$}
\put(285,172){$x_5$}
\put(285,232){$x_7$}

\put(255,5){$b$}
\put(235,37){$y_0$}
\put(255,67){$y_1$}

\put(235,97){$y_2$}
\put(255,127){$y_3$}
\put(235,157){$y_4$}
\put(255,187){$y_5$}
\put(235,217){$y_6$}
\put(255,247){$y_7$}
\put(255,280){$a$}

\put(228,-10){{${\cal F}_3$  ($t=6$)}}



\put(320,22){\circle{7}}
\put(320,85){\circle*{6}}
\put(320,145){\circle*{6}}
\put(320,205){\circle*{6}}

\put(345,5){$\otimes$}
\put(350,40){\circle*{6}}
\put(350,70){\circle*{6}}
\put(350,100){\circle*{6}}
\put(350,130){\circle*{6}}
\put(350,160){\circle*{6}}
\put(350,190){\circle*{6}}
\put(350,220){\circle*{6}}
\put(345,250){$\otimes$}

\put(380,55){\circle*{6}}
\put(380,115){\circle*{6}}
\put(380,175){\circle*{6}}
\put(380,237){\circle{7}}

\put(320,25){\line(0,1){180}}
\put(350,10){\line(0,1){240}}
\put(380,55){\line(0,1){180}}
\put(380,55){\line(-2,3){60}}
\put(380,115){\line(-2,3){60}}
\put(380,175){\line(-2,3){30}}
\put(350,40){\line(-2,3){30}}
\put(320,25){\line(2,3){60}}
\put(320,85){\line(2,3){60}}
\put(320,145){\line(2,3){60}}
\put(320,205){\line(2,3){30}}
\put(350,70){\line(2,3){30}}
\put(350,10){\line(2,3){30}}

\put(305,17){$x_0$}
\put(305,77){$x_2$}
\put(305,137){$x_4$}
\put(305,197){$x_6$}

\put(385,52){$x_1$}
\put(385,112){$x_3$}
\put(385,172){$x_5$}
\put(385,232){$x_7$}

\put(355,5){$b$}
\put(335,37){$y_0$}
\put(355,67){$y_1$}

\put(335,97){$y_2$}
\put(355,127){$y_3$}
\put(335,157){$y_4$}
\put(355,187){$y_5$}
\put(335,217){$y_6$}
\put(355,250){$a=y_7$}

\put(328,-10){{${\cal F}_4$  ($t=6$)}}

  \end{picture}
  \end{center}

\caption{The four families  of forbidden  weighted posets shown for $t=6$. Solid circles represent points with weight 2, hollow circles represent points with weight 1, and circles with a cross represent points whose weight may be either 1 or 2.}

\label{forb-2-fig}
 \end{figure}

   We make a few observations about these forbidden families.  When $t=0$, the poset in ${\cal F}_2$ is one of the posets shown in Figure~\ref{three-one-forb}.  Other than that exception, the weighted posets in $\cal F$ are distinct.  For general $t$, the poset in ${\cal F}_2$ has $2t+4$ points,  those in ${\cal F}_1$ and ${\cal F}_4$ have $2t+5$ points and the poset in ${\cal F}_3$ has $2t+6$ points.  The poset in ${\cal F}_2$ with  parameter $t+1$ is the dual of the poset in ${\cal F}_3$ with parameter $t$.  
    It is not hard to see that if one point is removed from any poset in $\cal F$, the resulting weighted poset has an interval representation ${\cal I} = \{I_x\}$ in which $|I_x| = f(x)$ for all $x$.  Thus the posets in $\cal F$ constitute a minimally forbidden set. 
 
 Since all weights in this section are positive,  we can remove the condition that a cycle contain a $(-)$ arc  from the minimality hypothesis, and record this in the following remark.  

 \begin{remark} {\rm
 If $P$ is the poset $(X,\prec)$ and $f: X \to \{1, 2\}$ is a weight function, then any cycle $C$ in $G'(P,f)$ with weight at most 0 will contain a $(-)$ arc.    }
 \label{min-hyp-rem}
 \end{remark}

We next state the main theorem of this section and prove
    that the first two statements are equivalent and that these imply the third.  The proof of the remaining part will be presented after a series of lemmas.
    
\begin{Thm}  Let $P$ be an interval order with $P = (X,\prec)$ and let $f:X \to \{1,2\}$ be a weight function.  The following are equivalent:
\begin{enumerate}
\item  $P$ has an interval representation in which $|I_x| = f(x)$ for all $x \in X$.
\item  $G(P,f)$ has no negative weight cycles.
\item  None of the  weighted posets in the set $\cal F$ of Definition~\ref{forb-def} are induced in  $(P,f)$.  
\end{enumerate}
\label{main-thm}
\end{Thm}
\noindent
\emph{Proof. }
The equivalence of $(1)$ and $(2)$ follows directly from Proposition~\ref{potential-prop}.



\medskip

\noindent $ (2) \Longrightarrow (3)$.    We show that for  each pair  $(Q,f)$ 
 in the forbidden set $\cal F$, the weighted digraph $G(Q,f)$ has a negative weight cycle.    For the poset $\threeone$  labeled $(a \succ y \succ b) \parallel x$ and weighted as shown in in Figure~\ref{three-one-forb}, the cycle
   $a \xrightarrow{ -f(y) - \epsilon} y \xrightarrow{-f(b) - \epsilon  }b  \xrightarrow{+f(b) } x   \xrightarrow{+ f(x) } a $  in $G(P,f)$ has weight $-f(y) -2\epsilon + f(x)  $, which is negative when either $f(y) = 2$ or $f(x) = 1$.  For the posets in families ${\cal F}_i$, for $1 \le i \le 4$, the cycles shown in Figure~\ref{cycle-family-fig}
  have weight 0 in $G'(P,f)$ and thus have negative weight in $G(P,f)$.     
 This completes the proof of $ (2) \Longrightarrow (3)$.

The next results establish properties of cycles in $G'(P,f)$ that satisfy the minimality hypothesis.  We include part (b) of Lemma~\ref{four-claims-lem} for completeness, although it is not needed in this paper.  By Remark~\ref{min-hyp-rem},  a cycle $C$ in  $G'(P,f)$ satisfies the  minimality hypothesis  for $(P,f)$ if  $wgt(C) \le 0$ and $C$   has the  minimum number of arcs among such  cycles.

\begin{lemma}
 If $C$ satisfies the minimality hypothesis for $(P,f)$ then the following hold.
 
\smallskip
\noindent 
(a) \   If  $S': a \xrightarrow{+}   {_b S_c}$  is a segment of  $C$  and $wgt(S) =0$ then $c \succ a$ in $P$ and cycle $C$ consists of $a \xrightarrow{+}{_b S_c} \xrightarrow{-}a$.

\smallskip
  \noindent
(b) \ 
  If  $S': {_a S_b} \xrightarrow{-f(c)} c$  is a segment of  $C$  and $wgt({_a S_b}) =0$ then the  cycle $C$ consists of  ${_a S_b} \xrightarrow{-} c \xrightarrow{} a.$

  \smallskip
  \noindent
  (c) \ If  $S': a \xrightarrow{-}{_b S_c} $  is a segment of  $C$  with  $wgt(S') =0$ and $S \neq \emptyset$  then $a=c$ and cycle $C$  is the segment $S'$.
  
\smallskip

\noindent
(d) \ 
  If  $S': {_a S_b}  \xrightarrow{+} c$  is a segment of  $C$   with $wgt(S') =0$  and $S \neq \emptyset$ then $a=c$ and cycle $C$  is the segment $S' $.
 
 \label{four-claims-lem}

  \end{lemma}

  \begin{proof}

To prove (a), suppose 
  ${_b S_c}$ is a segment of $C$ with $wgt({_b S_c} )= 0$.  If $a \parallel c$ or $a \succ c$ we can replace segment $S'$ on $C$ by the arc $a \xrightarrow{ }{ c} $ to get a shorter cycle whose weight is at most 0, a contradiction.    Hence $c \succ a$.  Then the cycle $a \xrightarrow{+f(a)}{_b S_c} \xrightarrow{-f(a)}a$ has weight at most 0 and thus is the cycle $C$ by minimality.   
  \smallskip
  
  To prove (b), 
  suppose  $S'$ is a segment of $C$ and $wgt({_a S_b}) =0$.  If $a \succ c$ we can replace $S'$ by $a \xrightarrow{-f(c)}{ c} $ on $C$  to obtain a shorter cycle  with the same weight in $G'(P,f)$ as $C$,   a contradiction.    Thus $c \succ a$ or $c \parallel a$ in $P$.  Now the cycle $  {_a S_b} \xrightarrow{-f(c)} c \xrightarrow{ } a$ in $G'(P,f)$ has the following weight:  $wgt({_a S_b}) - f(c) + w'_{ca} \le 0$.  By the minimality of $C$, this cycle is $C$, proving (b).

 \smallskip
 To prove (c), suppose $S'$ is the segment  $a \xrightarrow{-}{_b S_c} $ of $C $ with $wgt(S') = 0$ and $S \neq \emptyset$.   Since $wgt(S') =0$, we know $wgt({_b S_c}) = f(b)$.   If $b \succ c$ or $b \parallel c$ we can replace segment $S'$ by $a \xrightarrow{-f(b) }{ b}  \xrightarrow{ }{ c} $ to obtain a shorter cycle whose weight is negative (if $b \succ c$) or zero (if $b \parallel c$), contradicting the minimality of $C$.  Thus $c \succ b$.  Now  the cycle ${_b S_c} \xrightarrow{-f(b)}b $ has weight 0, and by the minimality of $C$ this is cycle $C$ and $a=c$.  This proves (c).

\smallskip

Finally, to prove (d), 
 suppose $S'$ is the  segment  ${_a S_b}  \xrightarrow{+} c$ of $C$ with $wgt(S') =0$ and $S \neq \emptyset$.  Since $wgt(S') =0$ and $w'_{bc} = f(b)$, we know $wgt(S) = -f(b)$.  If $a \succ b$ we can replace $S'$ by $a \xrightarrow{ -f(b)}{ b}  \xrightarrow{ +f(b)}{ c} $ to get a shorter  weight 0 cycle, a contradiction.   If $b \parallel a$ or $b \succ a$, the cycle $C':   {_a S_b}  \xrightarrow{ } a$ has weight zero  (if $b \parallel a$) or negative weight (if $b \succ a$).  By the minimality of $C$, we have $a=c$ and $C=C'$, proving (d).
\end{proof}

    \begin{lemma}
If $C$ satisfies the minimality hypothesis for $(P,f)$ then 
  $C$ consists of a path of $(+)$ arcs followed by a path of $(-)$ arcs.
  \label{plus-minus-lem}
\end{lemma}
 
 \begin{proof}
 Choose a starting point for $C$ so that the arcs of $C$ can be partitioned into segments $S_1, S_2, S_3, \ldots S_{2t}$ where the arcs in $S_i$ are $(+)$ for $i$ odd and the arcs in $S_i$ are $(-)$ for $i$ even.  We wish to show $t=1$.  For a contradiction, assume $t \ge 2$.  
 
 If $|wgt(S_{2j})| > |wgt(S_{2j-1})|$ for each $j$, then $wgt(C) \le -t \le -2$.  This contradicts Lemma~\ref{wgt-C-lem} when $r=2$.  Thus we can re-index the segments if necessary so that $|wgt(S_{2})| \le |wgt(S_3)|$.  Let $(a,b)$,  $(b,c)$ be the last two arcs in $S_1$, $(c,d)$ be the first arc of $S_2$, and $y$ the vertex at the end of $S_3$.  Thus cycle $C$ contains the  segment   $S$:   $a \xrightarrow{ +f(a)}b\xrightarrow{+f(b) }c  \xrightarrow{-f(d)}d$.
 
 First we show $f(b) \neq f(d)$.  For a contradiction, assume $f(b) = f(d)$.  Apply Lemma~\ref{four-claims-lem}(a)    to $S$  to conclude that $C$ is the cycle $a \xrightarrow{ +f(a)}b\xrightarrow{+f(b) }c  \xrightarrow{-f(d)}d \xrightarrow{-f(a)}a$, contradicting our assumption that $t \ge 2$.  
    Thus $f(b) \neq f(d)$. 
 
Consider the case in which $f(d) = 2$, and thus $f(b) = 1$.  Since $|wgt(S_{2})| \le |wgt(S_3)|$, we have $wgt({_c S_y}) \ge 0$.  Also, $w_{cd} = -f(d) = -2$, and arc weights belong to the set $\{-2, -1, 1, 2\}$ so there exists a vertex $x$ in segment $S_3$ for which $wgt({_c S_x}) \in \{0, -1\}$.    If $wgt({_c S_x}) = 0$, apply Lemma~\ref{four-claims-lem}(c) to  conclude that $C$ consists of $ {_c S_x} \to c$, contradicting $t \ge 2$.  If $wgt({_c S_x}) = -1$ then $wgt({_b S_x}) = 0$.  In this case, apply Lemma~\ref{four-claims-lem}(d) to the segment $ {_b S_x}$ to conclude that $C$ consists of $ {_c S_x} \to b$, contradicting $t \ge 2$.

  Otherwise, $f(d) = 1$, and thus $f(b) = 2$.  Again, $wgt({_c S_y}) \ge 0$  
  and arc weights belong to the set $\{-2, -1, 1, 2\}$, so there exists a vertex $x$ in segment $S_3$ for which $wgt({_c S_x}) \in \{0,  -1\}$.    If $wgt({_c S_x}) = 0$, apply Lemma~\ref{four-claims-lem}(c) to segment $ c\xrightarrow{-}{_d S_x}$ to conclude that $C$ consists of $ {_c S_x} \to c$, contradicting $t \ge 2$.  If $wgt({_c S_x}) = -1$ then $wgt({_d S_x}) = 0$.  In this case, apply Lemma~\ref{four-claims-lem}(d)  to the segment $ {_d S_x}$ to  contradict  $t \ge 2$.   
  \end{proof}

\begin{lemma}
Let $(P,f)$ be a weighted poset that does not contain  a weighted poset from Figure~\ref{three-one-forb}.
  If  cycle $C$ satisfies the minimality hypothesis for $(P,f)$ then  $C$ belongs to one of the four cycle   families $\mathcal{C}_i$ shown in Figure~\ref{cycle-family-fig}.
\label{families-lem}
\end{lemma}

\begin{proof}

As a result of Lemma~\ref{plus-minus-lem} we have shown that $C$ can be partitioned into two segments,  $ {_a S_b} $ and $  {_b T_a} $ where the arcs of $S$ are all $(-)$ and the arcs of $T$ are all $(+)$.    The weight on the arc entering $b$ is $-f(b)$ and the weight on the arc leaving $b$ is $+f(b)$.  This is indicated by  $\xrightarrow{ -}b \xrightarrow{ +}$ in Figure~\ref{cycle-family-fig}.  We consider cases based on the weight of the second arc of $T$.  We first consider the case in which the second arc of $T$ has weight $+2$, and show this results in $C $ belonging to $\mathcal{C}_1$ or $\mathcal{C}_2$ of Figure~\ref{cycle-family-fig}.

 Label segment $T$ as $b \xrightarrow{ +f(b)}x_1\xrightarrow{+2  }x_2  \xrightarrow{+ }x_3  \cdots$ and the segment $S$ as $\cdots \xrightarrow{ -}y_2\xrightarrow{-  }y_1  \xrightarrow{- }y_0 \xrightarrow{- }b$.    
We begin by showing that $wgt(y_1,y_0)= -1$, which will imply that $f(y_0) = 1$.  If $wgt(y_1,y_0) = -2$ then segment $ {_{y_1} C_{x_2}}$ has weight 0 and by Lemma~\ref{four-claims-lem}(c) we have $y_1=x_2$ and $C$ consists of  $y_1 \xrightarrow{ -2}y_0\xrightarrow{-  }b  \xrightarrow{+ }x_1 \xrightarrow{+2 }x_2 = y_1$.  Then  the points $y_1,y_0, b, x_1$ induce in $(P,f)$ a $\threeone$ in which $y_0$,  the middle element of the chain, has weight 2,  a contradiction.    Thus $wgt(y_1,y_0)= -1$.

Next suppose $wgt(x_k,x_{k+1}) = 2$ and $wgt(y_{k},y_{k-1}) = -2$ for all $k \ge 2$.  By Lemma~\ref{wgt-C-lem} with $r=2$, we know that  $C$ has weight $0$ or $-1$.  Since $wgt(y_1,y_0)= -1$, we know $wgt(C) = -1$ and $y_{s+1} = a = x_{s+1}$ for some $s \ge 2$.  Now segment ${_{x_1} C_{y_1}}$  has weight 0 and $wgt(b,x_1) > 0$, so by Lemma~\ref{four-claims-lem}(a), $C$  is the cycle $b \xrightarrow{ +}{_{x_1} C_{y_1}}\xrightarrow{-  }b   $, a contradiction since $(y_1,y_0)$ is an arc of $C$.  

Thus there exists some smallest $j \ge 2$ for which $wgt(x_j,x_{j+1}) = 1$  or $wgt(y_{j},y_{j-1}) = -1$.  If  $wgt(y_{j},y_{j-1}) = -1$ then segment ${_{y_{j}} C_{x_j}}$ has weight 0, and by Lemma~\ref{four-claims-lem}(c) this is all of $C$ and $C$ is in  $\mathcal{C}_1 $with $t = j-1$.    Otherwise,  $wgt(y_{j},y_{j-1}) = -2$ and 
$wgt(x_j,x_{j+1}) = 1$.  Now the segment ${_{y_{j}} C_{x_{j+1}}}$ has weight 0, so by Lemma~\ref{four-claims-lem}(c), $y_j = x_{j+1}$ and this is all of $C$.  Hence $C$ is in $\mathcal{C}_2$ with $t = j-1$.   

 We omit the details of the case in which the second arc of $T$ has weight 1, which  is similar.  In this case we conclude that  $C$ belongs to $\mathcal{C}_3$ or $\mathcal{C}_4$.     
 \end{proof}

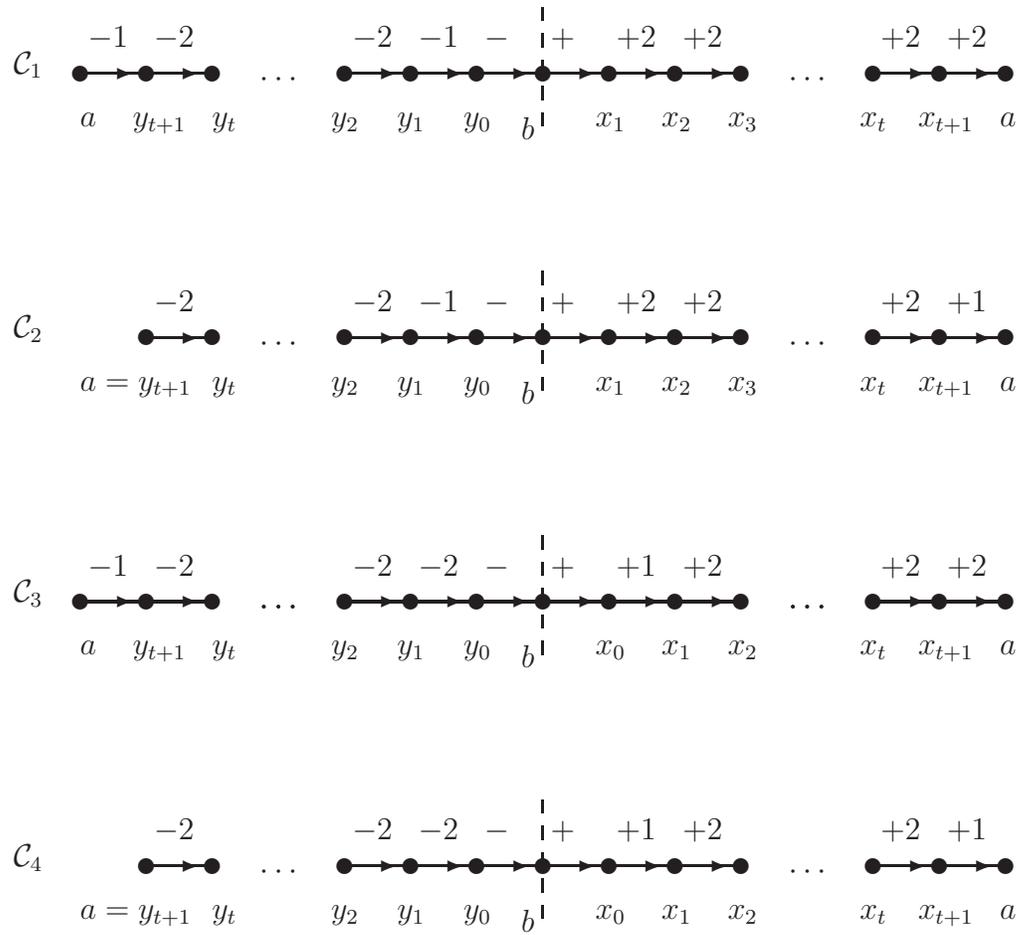
\begin{figure}
\begin{center}
 \begin{picture}(350,335)(0,0)
\thicklines

 
  \put(-25,325){$\mathcal{C}_1$}
 
 \put(0,325){\circle*{6}}
\put(25,325){\circle*{6}}
\put(50,325){\circle*{6}}
\put(100,325){\circle*{6}}
\put(125,325){\circle*{6}}
\put(150,325){\circle*{6}}
\put(175,325){\circle*{6}}
\put(200,325){\circle*{6}}
\put(225,325){\circle*{6}}
\put(250,325){\circle*{6}}
 \put(300,325){\circle*{6}}
\put(325,325){\circle*{6}}
\put(350,325){\circle*{6}}

\put(0,325){\vector(1,0){20}}
\put(20,325){\line(1,0){5}}

\put(25,325){\vector(1,0){20}}
\put(45,325){\line(1,0){5}}

\put(100,325){\vector(1,0){20}}
\put(120,325){\line(1,0){5}}

\put(125,325){\vector(1,0){20}}
\put(145, 325){\line(1,0){5}}

\put(150, 325){\vector(1,0){20}}
\put(170, 325){\line(1,0){5}}

\put(175, 325){\vector(1,0){20}}
\put(195, 325){\line(1,0){5}}

\put(200, 325){\vector(1,0){20}}
\put(220, 325){\line(1,0){5}}

\put(225, 325){\vector(1,0){20}}
\put(245, 325){\line(1,0){5}}
  
  \put(300, 325){\vector(1,0){20}}
\put(320, 325){\line(1,0){5}}

\put(325, 325){\vector(1,0){20}}
\put(345, 325){\line(1,0){5}}

  \put(3, 335){$-1$}
 \put(28, 335){$-2$}
\put(103, 335){$-2$}
\put(128, 335){$-1$}
\put(153, 335){$-$}
\put(178, 335){$+$}
\put(203, 335){$+2$}
\put(228, 335){$+2$}

\put(303, 335){$+2$}
\put(328, 335){$+2$}

  \put(68,320){$\cdots$}
  \put(268,320){$\cdots$}

 \put(68,120){$\cdots$}
  \put(268,120){$\cdots$}
 
\put(0,305){$a$}
\put(20,305){$y_{t+1}$}
\put(50,  305){$y_{t}$}
\put(95,  305){$y_{2}$}
\put(120,  305){$y_{1}$}
\put(145,  305){$y_{0}$}
\put(167, 300){$b$}
\put(195,  305){$x_1$}
\put(220,  305){$x_2$}
\put(245,  305){$x_3$}
\put(295,  305){$x_t$}
\put(317,  305){$x_{t+1}$}
\put(348,  305){$a$}

\multiput(175,305)(0,10){5}{\line(0,1){5}}

 
 
   \put(-25,225){$\mathcal{C}_2$}
\put(25,225){\circle*{6}}
\put(50,225){\circle*{6}}
\put(100,225){\circle*{6}}
\put(125,225){\circle*{6}}
\put(150,225){\circle*{6}}
\put(175,225){\circle*{6}}
\put(200,225){\circle*{6}}
\put(225,225){\circle*{6}}
\put(250,225){\circle*{6}}
 \put(300,225){\circle*{6}}
\put(325,225){\circle*{6}}
\put(350,225){\circle*{6}}


\put(25,225){\vector(1,0){20}}
\put(45,225){\line(1,0){5}}

\put(100,225){\vector(1,0){20}}
\put(120,225){\line(1,0){5}}

\put(125,225){\vector(1,0){20}}
\put(145, 225){\line(1,0){5}}

\put(150, 225){\vector(1,0){20}}
\put(170, 225){\line(1,0){5}}

\put(175, 225){\vector(1,0){20}}
\put(195, 225){\line(1,0){5}}

\put(200, 225){\vector(1,0){20}}
\put(220, 225){\line(1,0){5}}

\put(225, 225){\vector(1,0){20}}
\put(245, 225){\line(1,0){5}}
  
  \put(300, 225){\vector(1,0){20}}
\put(320, 225){\line(1,0){5}}

\put(325, 225){\vector(1,0){20}}
\put(345, 225){\line(1,0){5}}

 \put(28, 235){$-2$}
\put(103, 235){$-2$}
\put(128, 235){$-1$}
\put(153, 235){$-$}
\put(178, 235){$+$}
\put(203, 235){$+2$}
\put(228, 235){$+2$}

\put(303, 235){$+2$}
\put(328, 235){$+1$}

  \put(68,220){$\cdots$}
  \put(268,220){$\cdots$}

 \put(68,120){$\cdots$}
  \put(268,120){$\cdots$}
 
\put(0,205){$a=y_{t+1}$}
\put(50,  205){$y_{t}$}
\put(95,  205){$y_{2}$}
\put(120,  205){$y_{1}$}
\put(145,  205){$y_{0}$}
\put(167,200){$b$}
\put(195,  205){$x_1$}
\put(220,  205){$x_2$}
\put(245,  205){$x_3$}
\put(295,  205){$x_t$}
\put(317,  205){$x_{t+1}$}
\put(348,  205){$a$}

\multiput(175,205)(0,10){5}{\line(0,1){5}}

 
    \put(-25,125){$\mathcal{C}_3$}
 \put(0,125){\circle*{6}}
\put(25,125){\circle*{6}}
\put(50,125){\circle*{6}}
\put(100,125){\circle*{6}}
\put(125,125){\circle*{6}}
\put(150,125){\circle*{6}}
\put(175,125){\circle*{6}}
\put(200,125){\circle*{6}}
\put(225,125){\circle*{6}}
\put(250,125){\circle*{6}}
 \put(300,125){\circle*{6}}
\put(325,125){\circle*{6}}
\put(350,125){\circle*{6}}

\put(0,125){\vector(1,0){20}}
\put(20,125){\line(1,0){5}}

\put(25,125){\vector(1,0){20}}
\put(45,125){\line(1,0){5}}

\put(100,125){\vector(1,0){20}}
\put(120,125){\line(1,0){5}}

\put(125,125){\vector(1,0){20}}
\put(145,125){\line(1,0){5}}

\put(150,125){\vector(1,0){20}}
\put(170,125){\line(1,0){5}}

\put(175,125){\vector(1,0){20}}
\put(195,125){\line(1,0){5}}

\put(200,125){\vector(1,0){20}}
\put(220,125){\line(1,0){5}}

\put(225,125){\vector(1,0){20}}
\put(245,125){\line(1,0){5}}
  
  \put(300,125){\vector(1,0){20}}
\put(320,125){\line(1,0){5}}

\put(325,125){\vector(1,0){20}}
\put(345,125){\line(1,0){5}}

  \put(3,135){$-1$}
 \put(28,135){$-2$}
\put(103,135){$-2$}
\put(128,135){$-2$}
\put(153,135){$-$}
\put(178,135){$+$}
\put(203,135){$+1$}
\put(228,135){$+2$}

\put(303,135){$+2$}
\put(328,135){$+2$}

 \put(68,120){$\cdots$}
  \put(268,120){$\cdots$}
  \put(0, 105){$a $}
\put(20, 105){$y_{t+1}$}
\put(50, 105){$y_{t}$}
\put(95, 105){$y_{2}$}
\put(120, 105){$y_{1}$}
\put(145, 105){$y_{0}$}
\put(167,100){$b$}
\put(195, 105){$x_0$}
\put(220, 105){$x_1$}
\put(245, 105){$x_2$}
\put(295, 105){$x_t$}
\put(317, 105){$x_{t+1}$}
\put(348, 105){$a$}

\multiput(175,105)(0,10){5}{\line(0,1){5}}

 
 
    \put(-25,25){$\mathcal{C}_4$}
    
\put(25,25){\circle*{6}}
\put(50,25){\circle*{6}}
\put(100,25){\circle*{6}}
\put(125,25){\circle*{6}}
\put(150,25){\circle*{6}}
\put(175,25){\circle*{6}}
\put(200,25){\circle*{6}}
\put(225,25){\circle*{6}}
\put(250,25){\circle*{6}}
 \put(300,25){\circle*{6}}
\put(325,25){\circle*{6}}
\put(350,25){\circle*{6}}

\put(25,25){\vector(1,0){20}}
\put(45,25){\line(1,0){5}}

\put(100,25){\vector(1,0){20}}
\put(120,25){\line(1,0){5}}

\put(125,25){\vector(1,0){20}}
\put(145,25){\line(1,0){5}}

\put(150,25){\vector(1,0){20}}
\put(170,25){\line(1,0){5}}

\put(175,25){\vector(1,0){20}}
\put(195,25){\line(1,0){5}}

\put(200,25){\vector(1,0){20}}
\put(220,25){\line(1,0){5}}

\put(225,25){\vector(1,0){20}}
\put(245,25){\line(1,0){5}}
  
  \put(300,25){\vector(1,0){20}}
\put(320,25){\line(1,0){5}}

\put(325,25){\vector(1,0){20}}
\put(345,25){\line(1,0){5}}
 
 \put(28,35){$-2$}
\put(103,35){$-2$}
\put(128,35){$-2$}
\put(153,35){$-$}
\put(178,35){$+$}
\put(203,35){$+1$}
\put(228,35){$+2$}

\put(303,35){$+2$}
\put(328,35){$+1$}
 
 \put(68,20){$\cdots$}
  \put(268,20){$\cdots$}
\put(0,5){$a= y_{t+1}$}
\put(50,5){$y_{t}$}
\put(95, 5){$y_{2}$}
\put(120, 5){$y_{1}$}
\put(145, 5){$y_{0}$}
\put(167,0){$b$}
\put(195, 5){$x_0$}
\put(220, 5){$x_1$}
\put(245, 5){$x_2$}
\put(295, 5){$x_t$}
\put(317, 5){$x_{t+1}$}
\put(348, 5){$a$}

\multiput(175,5)(0,10){5}{\line(0,1){5}}

  \end{picture}
  \end{center}

\caption{The four families  $\mathcal{C}_i$ of cycles.}

\label{cycle-family-fig}
 \end{figure}

 \begin{lemma}
Suppose $C$ satisfies the minimality hypothesis for $(P,f)$.
  If $C$ is a member of  $\mathcal{C}_i$ of Figure~\ref{cycle-family-fig} for some $i: 1 \le i \le 4$ and some $t \ge 0$ then a weighted poset  from ${\cal F}_i$  is induced in $(P,f)$. 
 \label{cycle-to-poset-lem}
 \end{lemma}
 
 \begin{proof}
 By the definition of $G'(P,f)$, we know  that in  poset $P$ we have $y_{t+1} \succ y_t \succ \cdots  \succ y_1 \succ y_0 \succ b$,   $x_i \parallel x_{i+1}$ for $1 \le i \le t$, and $x_{t+1} \parallel a$.  Likewise, $a \succ y_{t+1}$ if $C$ belongs to $\mathcal{C}_1$ or $\mathcal{C}_3$,
 $b \parallel x_1$ if $C$ belongs to $\mathcal{C}_1$ or $\mathcal{C}_2$, and 
  $b \parallel x_0$ and $x_0 \parallel x_1$ if  $C$ belongs to $\mathcal{C}_3$ or $\mathcal{C}_4$.    If $C$ belongs to $\mathcal{C}_3$ or $\mathcal{C}_4$ and $b \succ x_1$ or $b \parallel x_1$, then replacing segment 
 $b \xrightarrow{  }x_0\xrightarrow{  }x_1  $ by the arc  $b \xrightarrow{  } x_1  $ in $C$ results in a shorter cycle with weight at most 0, a contradiction.    Hence $b \prec x_1$ when $C$ belongs to $\mathcal{C}_3$ or $\mathcal{C}_4$.  Similarly, 
   $y_0 \prec x_2$ since $f(x_1) = 2$ and if $y_0 \succ x_2$ or $y_0 \parallel x_2$, replacing segment ${_{y_0} S_{x_2}}$ by  the arc $y_0 \to x_2$ would result in a shorter cycle with weight at most 0, a contradiction.

   It remains to show the following relationships exist in $P$ between points of $\mathcal{C}$:  (i) $x_i \prec x_{i+2}$ for $0 \le i \le t-1$, \ (ii) $y_i \prec x_{i+2}$ for $0 \le i \le t-1$, \  (iii) $x_i \prec y_{i+1}$ for $1 \le i \le t$, \  (iv) $x_i \parallel  y_i$ for $1 \le i \le t+1$,  and (v) $x_i \parallel  y_{i-1}$ for $1 \le i \le t+1$.
 First consider a segment ${_u S_v}$ of $C$ with $wgt({_u S_v}) \ge 3$ and $S \neq \emptyset$.    If $u \succ v$ or $u \parallel v$ for any $i$ then we could replace the segment $S$ by  the arc $u \to v$  to obtain a shorter cycle whose weight is at most 0, contradicting the minimality of $C$.  Thus whenever a segment ${_u S_v}$ of $C$ has $wgt({_u S_v}) \ge 3$, we can conclude that $u \prec v$ in $P$.  This immediately implies that (i) $x_i \prec x_{i+2}$  for all $i \ge 0$,  (ii) $y_i \prec x_{i+2}$ for all $i \ge 1$, and (iii) $x_i \prec y_{i+1}$ for $i \ge 1$.    
 
 Next we show (iv) $x_i \parallel y_i$ for  all $i \ge 0$.  In each of the four families, for each $i$ we have $wgt({_{y_i} C_{x_i}}) = -1$ and  $wgt({_{x_i} C_{y_i}}) = 1$.  If $y_i \succ x_i$ we could replace the non-trivial segment ${_{y_i} C_{x_i}}$ by  the arc $y_i \to x_i$ to get a shorter cycle with weight at most 0, a contradiction.  If $x_i \succ y_i$, the segment ${_{x_i} C_{y_i}}$  is non-trivial and similarly can be replaced by $x_i \to y_i$ to get a   contradiction.    Thus $x_i \parallel y_i$ for  all $i \ge 0$. 
 
 Finally,  we show (v).  For $i > 1$, in each family, and for $i \ge 1$ in families 3 and 4, $wgt({_{y_{i-1}} C_{x_i}}) = 1$  and $wgt({_{x_i} C_{y_{i-1}}}) = -1$.  Hence by the argument above, $x_i \parallel y_{i-1}$.  In families 1 and 2,  $wgt({_{y_0} C_{x_1}}) = wgt({_{x_1} C_{y_0}}) = 0$ and replacing  $_{y_0} C_{x_1}$ by  the arc $y_0 \to x_1$  if $y_0 \succ x_1$ or replacing ${_{x_1} C_{y_0}}$ by the arc  $x_1 \to y_0$ if $x_1 \succ y_0$ results in a shorter cycle of weight 0, a contradiction.  Hence $x_i \parallel y_{i-1}$ for all $i \ge 1$.
 \end{proof}

We now have the tools to complete the proof of Theorem~\ref{main-thm}.

\begin{proof}

\noindent $ (3) \Longrightarrow (2)$.
Let $P$ be a poset with $P = (X, \prec)$,  and $f$ be a weight function  $f:X \to \{1,2\}$ for which none of the weighted posets in $\cal F$  of Definition~\ref{forb-def} is induced in   $(P,f)$.   

 For a contradiction,  assume  the weighted digraph $G(P,f)$ has a negative cycle, so by Proposition~\ref{G'-rem}, the weighted digraph $G'(P,f)$ has a cycle whose weight is at most 0.   Let $C$ be such a cycle in $G'(P,f)$ with the minimum number of arcs, thus $C$ satisfies the  minimality hypothesis  of Definition~\ref{min-hyp-def}.
 By Lemma~\ref{families-lem}, 
 the cycle $C$ belongs to one of the four cycle families of Figure~\ref{cycle-family-fig}.  Now by Lemma~\ref{cycle-to-poset-lem}, 
 $(P,f)$ contains an induced weighted poset from ${\cal F}_i$ for some $i$, a contradiction.
 \end{proof}   
  
   As described at the end of Section~\ref{0-1-sec},
   our results give a polynomial time certifying algorithm to determine whether a weighted poset  $(P,f)$, with weights in the set $\{1,2\}$,   has an interval representation ${\mathcal I} = \{I_x: x \in X\}$ in which $|I_x| = f(x)$ for each point $x$.  In the affirmative case, such a  representation can be obtained  in polynomial time as described in the proofs of     Proposition~\ref{potential-prop} and Theorem~\ref{potential-no-neg}.
   Otherwise a negative cycle with a minimum number of arcs is detected and a corresponding forbidden weighted poset from 
   Theorem~\ref{main-thm} can be found.
   
   \section{Conclusion}
   
   In this paper we  use digraph methods to find interval representations of posets in which there are two permissible interval lengths, either  $\{0,1\}$ or $\{1,2\}$.    We characterize those posets that have an interval representation in which the interval lengths are in the set $\{0,1\}$.  We do not have an analogous forbidden poset characterization when the set of permissible lengths is $\{1,2\}$.  Indeed, we do not expect to find an analogue of Lemma~\ref{length-lem}, which allowed us to determine the interval length corresponding to each point.
Instead,  our characterization in Section~\ref{1-2-sec}
  involves  posets in which each point has a weight of $1$ or $2$,   and these weights correspond  to   pre-specified interval lengths.

  Similar digraph methods can be used to provide   efficient algorithms to check for a representation when each interval has length between a specified lower and upper bound (see \cite{Is09}).   In \cite{BoIsTr-2} we use digraph methods to give a simple proof of the theorem due to Fishburn \cite{Fi83} characterizing interval orders that have a representation in which all interval lengths are between 1 and a fixed integer $k$.
  However, these methods do not appear to extend to instances when  set of possible interval lengths is not connected.

   When there are two permissible interval  lengths and one length is 0, by scaling we may assume the other length is 1. So there are no further cases to consider. Also by scaling, when there are two permissible interval lengths and neither is 0 we may assume that the smaller is 1. It would be natural to next consider the larger length to be (k+1)/k (corresponding to lengths k and (k+1) or to consider the larger length to be k for some integer k.  While digraph models give efficient algorithms for recognizing whether an order belongs to these classes, we seek forbidden order characterizations analogous to Theorem~\ref{main-thm}.

\end{document}